\documentclass[11pt,letterpaper]{amsart}
\usepackage{amsmath}
\usepackage{amsfonts}
\usepackage{tikz-cd}
\usepackage{amsthm}
\usepackage{amssymb}
\usepackage{hyperref}
\usepackage{mathtools}
\usepackage{mathrsfs}
\usepackage{wasysym}
\usepackage{cite}
\usepackage[utf8]{inputenc} \usepackage[T1]{fontenc}

\theoremstyle{plain}
\newtheorem{theorem}{Theorem}           
\newtheorem{proposition}[theorem]{Proposition}

\newtheorem{lemma}[theorem]{Lemma}

\theoremstyle{definition}

\newtheorem{example}[theorem]{Example}

\theoremstyle{remark}
\newtheorem*{remark}{Remark}

\title{Failure of weak approximation in adjoint groups}
\author{Chayansudha Biswas}
\address{Department of Mathematics, University of Pennsylvania, Philadelphia, PA 19104-6395, USA}
\email{biswasch@sas.upenn.edu}

\begin{document}
    \maketitle

    \begin{abstract}
        Platonov conjectured (orally around 1978, with a published version appearing in 1991) that adjoint groups over arbitrary infinite fields are rational as varieties, and consequently have weak approximation. The rationality part of the conjecture was disproved by Merkurjev in 1996, and this led to the construction of adjoint groups failing weak approximation with respect to two independent valuations. We construct an infinite family of adjoint groups that fail weak approximation with respect to a single valuation.

    \end{abstract}

    \section*{Introduction}
    A generalization of the Chinese Remainder Theorem is the property of weak approximation of a field: if $S$ is any finite set of inequivalent discrete valuations on a field $F$, then the image of $F$ is dense in $\prod_{v\in S}F_v$ under the diagonal embedding, where $F_v$ denotes the completion of $F$ with respect to a valuation $v\in S$ \cite[Theorem 1.4]{MR1278263}. The topology considered here is the product adic topology. If $X$ is a variety defined over the field $F$, then one can ask if the image of $X(F)$ is dense in $\prod_{v\in S}X(F_v)$ under the diagonal embedding. If this is the case, then we say that $X$ \textit{has weak approximation with respect to} $S$. We say that \emph{$X$ has weak approximation} if it has weak approximation with respect to every finite set of discrete valuations of $F$. All rational varieties have weak approximation, and hence failure of weak approximation is an obstruction to rationality.

    Weak approximation in the setting of algebraic groups is classical and well understood in many cases. For instance, it is well established that all semisimple simply connected and semisimple adjoint groups over global fields have weak approximation. The same statement is known to be false (due to work by Platonov in \cite{MR568310}) over arbitrary valued fields in the simply connected case, motivating a focus on adjoint groups instead. Platonov conjectured all semisimple adjoint groups over arbitrary fields are rational as varieties, and therefore have weak approximation \cite[\S 7.3]{MR1278263}. The rationality part of the conjecture was disproved by Merkurjev in \cite{merkurjev1996r}, where examples were given of adjoint groups that fail rationality. Around the same time, over arbitrary fields, weak approximation was shown to hold for several cases of adjoint groups by Th\u{a}\'{n}g in \cite{thuang1996weak}. More examples of nonrational adjoint groups can be found in \cite{gille1997examples},  \cite{berhuy2004cohomological} and \cite{bhaskhar2014more}. 
    
    The obstruction to rationality in these examples is the failure of universal $R$-triviality: for a group $G$ defined over a field $F$, we say two points $x,y\in G(F)$ are \emph{$R$-equivalent} if there is a rational map $f\colon \mathbb{A}_F^1\dashrightarrow G$ defined at $0$ and $1$ such that $f(0) = x$ and $f(1) = y$. This is an equivalence relation on $G(F)$ and we denote by $G(F)/R$ the set of $R$-equivalence classes. We say the group $G$ is \emph{universally $R$-trivial} if $G(E)/R=1$ for all field extensions $E$ of $F$. This notion of $R$-equivalence was introduced by Manin in \cite{manin1986cubic}. Rationality of (the underlying variety of) a group implies universal $R$-triviality, and Merkurjev constructed adjoint groups that fail universal $R$-triviality to obtain nonrational adjoint groups. This construction also gives rise to adjoint groups that fail weak approximation with respect to two independent valuations, hence disproving the conjecture completely. 
    
    One could then ask if there exist adjoint groups failing weak approximation with respect to a single valuation. An example was proposed by Gille in \cite{MR3367427}, but it turned out to have an error. In the present article, we will settle the problem by constructing an infinite family of adjoint groups of classical type $D_n$ that fail weak approximation with respect to a single valuation (see Theorem \ref{conjecture-disproved}, and see Example \ref{example} for an explicit example). These appear to be the first known examples of this phenomenon.

    For every global field $F$ of characteristic different from $2$ in which $-1$ is a square, we will give a recipe to construct infinitely many adjoint groups, defined over a suitable overfield of $F$, that fail weak approximation with respect to a single valuation.

    To construct our counterexamples, we fix such a field $F$, and using a criterion of Merkurjev we first construct an adjoint group $G$ of type $D_n$ over~ $F$ such that $G(F)/R=1$ but $G(E)/R\neq 1$ for some extension $E$ over~ $F$ (Lemma \ref{efg}). The main trick is to slightly modify the situation above to obtain a field $K$ over $F$ which admits a discrete valuation $v$ such that $G(K)/R=1$ and $G(K_v)/R\neq 1$ (Lemma \ref{val-fld}).
    We then use that $R$-equivalence classes of $K_v$-rational points of $G$ are open in $G(K_v)$ in the $v$-adic topology (Theorem \ref{r-eq-open}) to get our counterexample. In particular, we consider $g\in G(K_v)$ not in the $R$-equivalence class of $1$. Weak approximation in $G$ would imply that the $R$-equivalence class of $g$ has nonempty intersection with $G(K)$. But $G(K)/R=1$ implies that points in this intersection are $R$-equivalent to the identity, which gives us a contradiction.

    \section{Constructing adjoint groups failing universal $R$-triviality}

    The aim of this section is the construction of adjoint groups that fail universal $R$-triviality. This is done in Proposition \ref{merk-psim}, using a criterion due to Merkurjev. Proposition \ref{merk-psim} will later be used in \S 2 to prove the main theorem. The following proposition is the main technical ingredient:
    
    \begin{proposition}\label{main-prop-sec1}
    Let $F$ be any global field with $\operatorname{char}F\neq 2$ in which $-1$ is a square. Then there is a quaternion algebra $Q$ over $F$ and an element~ $d\in~ F^{\times}\setminus F^{\times 2}$ such that $d\in -\operatorname{Nrd}(Q)$ and $d\notin\operatorname{Nrd}(Q^+)$. Here $Q^+$ denotes the subspace of $Q$ consisting of the pure quaternions, and $\operatorname{Nrd}\colon Q\to F$ is the reduced norm.
\end{proposition}

    For $F$ and $Q$ as in the proposition, let $A = M_n(Q)$, for $n\geq 3$ an odd integer. Then the hypothesis of \cite[Proposition 11]{merkurjev1996r} is satisfied, and thus there exists an involution $\sigma$ on $A$ with $\operatorname{disc}(\sigma) = d\cdot F^{\times 2}$. In the proof of this proposition it is shown that the hypothesis of \cite[Theorem 2]{merkurjev1996r} is satisfied, and hence the $F$-group $G = \mathbf{PSim}_+(A,\sigma)$ is an adjoint group of type $D_n$ failing universal $R$-triviality. 
    % (See \cite[p.2]{berhuy2004cohomological} for a definition of this group.) 
    For an $F$-algebra with involution $(A,\sigma)$, where the center of $A$ contains $F$, the group $\mathbf{PSim}_+(A,\sigma)$ is defined as follows:

    An element $a\in A$ is said to be a \emph{similitude} if $\sigma(a)\cdot a\in F^\times$. For an $F$-algebra $T$, define $\mathbf{Sim}(A,\sigma)(T): = \{a\in A_T: \sigma_T(a)\cdot a\in F^\times\otimes_F T^\times \}$. This is the algebraic group of similitudes. The algebraic group of projective similitudes is then defined to be the quotient $\mathbf{PSim}(A,\sigma): = \mathbf{Sim}(A,\sigma)/R_{Z/F}(\mathbb{G}_{m,Z})$, where $Z$ is the center of $A$ and $R_{Z/F}$ denotes the Weil restriction functor. We denote by $\mathbf{PSim}_+(A,\sigma)$ the connected component of the identity of $\mathbf{PSim}(A,\sigma)$.

    Failure of universal $R$-triviality means that there exists a field $E$ over $F$ such that $G(E)/R\neq 1$. The field $E$ is constructed from an iterated finite tower of fields over $F$ that are either function fields of quadrics or function fields of Severi-Brauer varieties (see \cite[Theorem 2]{merkurjev1996r}). This ensures that $E$ admits a finite transcendence basis $T$ over $F$ such that $E$ is a finite separable extension of $F(T)$.
    % This ensures that $E$ is a separably generated extension over $F$ of finite transcendence degree. 
    We will use this fact to prove in the next section that $G$ fails weak approximation over a suitable overfield of $F$. We summarize this in the following proposition:
    \begin{proposition}\label{merk-psim}
        For any global field $F$ with $\operatorname{char}F\neq 2$ in which $-1$ is a square, there is an adjoint $F$-group $G$ of classical type $D_n$, where $n$ is an odd integer $\geq 3$, and a finitely generated, separably generated extension $E$ over $F$ such that $G(E)/R\neq 1$.
    \end{proposition}

    \begin{remark}
        One checks, by following Merkurjev's construction, that the transcendence degree of $E$ over $F$ can be bounded above by $3n+4$ in the above notation.
    \end{remark}

    %   one checks by following \cite{merkurjev paper} and \ref{para after prop 1} that one can find such a group for which tr.deg_F(E) is bounded by 3n+4.

    The remainder of this section is devoted to the proof of Proposition \ref{main-prop-sec1}.

    \begin{lemma}\label{d-exists}
    Let $F$ be any global field of characteristic different from $2$, and denote by $\Omega_F$ the set of all places of $F$. Let $\Sigma$ be any finite subset of $\Omega_F$. Then there is an element $d\in F^{\times}\setminus F^{\times 2}$ such that $d\in F_v^{\times 2}$ for all $v\in \Sigma$.
\end{lemma}
\begin{proof}
    Let $w$ be a place of $F$ not lying in the set $\Sigma$, and consider any element $a = (a_v)_{v\in\Sigma\cup\{w\}}\in \prod_{v\in\Sigma\cup\{w\}}F_v$ satisfying: \begin{enumerate}
        \item $a_v = 1$ for $v\in \Sigma$;
        \item $a_w\in F_w^\times\setminus F_w^{\times 2}$.
    \end{enumerate}
    Since $F$ is a global field, $F_v^{\times 2}$ is a subgroup of finite index in $F_v$ for $v\in\Omega_F$. This, coupled with the fact that $\operatorname{char}F_v\neq 2$, implies that $F_v^{\times 2}$ is $v$-adically open (and closed) in $F_v^{\times}$. Therefore, there is a neighborhood $U\subset \prod_{v\in\Sigma\cup\{w\}}F_v$ containing $a$ such that for any $b = (b_v)\in U$, $b_v$ is a square in $F_v$ for $v\in \Sigma$ and $b_w$ is not a square in $F_w$.

    By weak approximation for $F$, there is an element $d\in F\cap U$. Therefore, $d$ is a square in $F_v$ for all $v\in \Sigma$, and $d$ is not a square in $F_w$. In particular, we conclude $d\notin F^{\times 2}$. This element $d$ satisfies the conditions in the lemma.
\end{proof}

Fix a global field $F$ with $\operatorname{char}F\neq 2$ in which $-1$ is a square.
Choose a finite subset $\Sigma$ of $\Omega_F$ of even cardinality containing some valuation $v$ such that the corresponding residue field $\kappa_v$ has characteristic different from $2$. It is always possible to do so, since for our choice of $F$, there are at most finitely many valuations having residue characteristic $2$.
Define $Q$ to be the quaternion algebra over $F$ that is ramified precisely at the places in $\Sigma$; by \cite[Theorem 2.7.5]{maclachlan2013arithmetic}, $Q$ is uniquely determined up to isomorphism. Fix also $d\in F^\times$ as in Lemma \ref{d-exists}. We will retain this notation for the remainder of this section.

\begin{lemma}\label{neg-norm}
    One has $d\in -\operatorname{Nrd}(Q)$.
\end{lemma}
\begin{proof}
    Since our choice of $F$ has cohomological dimension $2$ (see \cite[Proposition 8.3.17, Theorem 10.1.11]{MR1737196}), the reduced norm map $\operatorname{Nrd}$ is surjective by \cite[Theorem 8.9.1]{gille2017central}. This immediately implies the claim.
\end{proof}

\begin{lemma}\label{ab-vals}
    Given $v\in \Sigma$, there exist $a,b\in F$ with $Q\simeq (a,b)_F$ such that $v(b)=1$ and $v(a)=0$. 
\end{lemma}
\begin{proof}
    Let $(a,b)_F$ be any representation of $Q_F$. Let $\pi$ be a uniformizer corresponding to the place $v$. We may write $a = a_0\pi^{v(a)}, b= b_0\pi^{v(b)}$ for $v$-adic units $a_0,b_0$.
    We will first show that at least one of $v(a)$ and $v(b)$ is nonzero.
    
     Recall that for a place $u\in\Omega_F$ the Hilbert symbol at $u$ is defined as the function $h_u\colon F_u^\times\times F_u^\times\to \{\pm 1\}$, which sends the pair $(a,b)$ to $1$ if and only if the quaternion algebra $(a,b)_{F_u}$ is split, i.e., isomorphic to $\operatorname{M}_2(F_u)$.
     By \cite[12.4.8]{voight2021quaternion}, we have $h_v(a,b) = (-1)^{v(a)v(b)(q-1)/2}\left(\frac{a_0}{v}\right)^{v(b)}\left(\frac{b_0}{v}\right)^{v(a)}$, where $q$ is the cardinality of the residue field $\kappa_v$, and $\left(\frac{a_0}{v}\right)$ and $\left(\frac{b_0}{v}\right)$ denote the Legendre symbols. If $v(a) = 0 = v(b)$, then $h(a,b) = 1$, which would mean that $(a,b)_{F_v}$ is split, a contradiction to the fact that $Q$ is ramified at $v$.

    Since $(a,b)_F\simeq (b,a)_F$, we may assume without loss of generality that $v(b)\neq 0$. Throughout the rest of the proof, we will frequently use the fact that scaling $a$ or $b$ by the square of a unit in $F$ does not change the isomorphism class of $(a,b)_F$.
    If $v(b)<0$, then $v(\pi^{-2v(b)}b) = -v(b)>0$. Therefore, by replacing $b$ by $\pi^{-2v(b)}b$ if necessary, we have $Q\simeq (a,b)_F$ with $v(b)>0$.

    If $v(a)<0$, replace $a$ by $a\pi^{-2v(a)}$ to ensure $v(a)\geq 0$. Assume without loss of generality that $v(a)\leq v(b)$ (since $(a,b)_F\simeq (b,a)_F)$). Recall that we have already shown in the first part of the proof that the case $(a,b)_F\simeq (a_0,b_0)_F$ is not possible. So by scaling each argument by a suitable power of $\pi^{-2}$, we may assume that $(a,b)_F$ is isomorphic to either $(a_0,b_0\pi)_F$ or $(a_0\pi,b_0\pi)$.  In the first case, simply replace $a$ by $a_0$ and $b$ by $b_0\pi$. In the second case, use the isomorphism $(x,y)_F\simeq (-xy,y)_F$ to obtain $(a,b)_F\simeq (a_0\pi,b_0\pi)_F\simeq (-a_0b_0\pi^2,b_0\pi)_F\simeq (-a_0b_0,b_0\pi)_F$. Now replace $a$ by $-a_0b_0$ and $b$ by $b_0\pi$.
\end{proof}

\begin{lemma}\label{not-pure-quat-norm}
    The element $d\in F^\times$ is not the norm of a pure quaternion in $Q$.
\end{lemma}
\begin{proof}
    Fix a valuation $v\in \Sigma$ with uniformizer $\pi$ such that the residue field $\kappa_v$ has characteristic different from $2$. By Lemma \ref{ab-vals}, there exist $a,b\in F$ with $Q\simeq (a,b)_F$ and such that $v(a) = 0$ and $v(b)=1$.
    Suppose $d = \operatorname{Nrd}(xI + yJ + zK) = -ax^2-by^2+abz^2$ for some $x,y,z\in F$. By multiplying throughout by a suitable $n$-th power of $\pi^2$, we obtain an equation of the form
    \begin{equation}\label{normform}
        D = -aX^2-bY^2 + abZ^2,
    \end{equation} where $D = d\pi^{2n}$, and $X = x\pi^{n},Y = y\pi^{n}$ and $Z = z\pi^{n}$ are in $\mathcal{O}_v$ and at least one of them is not in $\mathfrak{m}_v$. 
    
    Reducing the equation modulo $\mathfrak{m}_v$, we obtain $D\equiv -aX^2(\operatorname{mod}\mathfrak{m}_v)$. First assume $X\in\mathcal{O}_v^\times$. Since $d$ is a square in $F_v$, so is $D$, and let $C\in F_v$ be such that $C^2 = D$. Then we find $(iC/X)^2\equiv a(\operatorname{mod}\mathfrak{m}_v)$, which would mean that $a$ is a square modulo $\mathfrak{m}_v$, which contradicts the fact that $Q_{F_v}$ is not split. (Here $i$ is a square root of $-1$ in $F$.)

    So $v(X)>0$, and either $Y$ or $Z$ has zero valuation. Suppose it is $Y$. The case $Z\in\mathcal{O}_v^\times$ is similar. Since $v(b) = 1$, let $b = b_0\pi$ for $b_0\in\mathcal{O}_v^\times$. The congruence $D\equiv -aX^2\operatorname{mod}\mathfrak{m}_v$ gives $D\equiv X^2\equiv 0\operatorname{mod}\mathfrak{m}_v$. Let $D = D_0\pi^k, X= X_0\pi^j$ for $D_0,X_0\in\mathcal{O}_v^\times$ and $k,j\geq 1$. On the other hand, since $D = d\pi^{2n}$, one has $k = v(D) = v(d) + 2n$, which is even because $d$ is a square in $F_v$ (and hence $v(d)\in 2\mathbb{Z}$). Therefore, $k\geq 2$. Substituting these in (\ref{normform}), we obtain:\begin{equation}
        D_0\pi^k = -aX_0^2\pi^{2j} - b_0\pi Y^2 + ab_0\pi Z^2.
    \end{equation}
    Dividing throughout by $\pi$, we get
    \begin{equation}
        D_0\pi^{k-1} = -aX_0^2\pi^{2j-1} - b_0Y^2 + ab_0Z^2;
    \end{equation}
    and reducing this modulo $\mathfrak{m}_v$ gives:\begin{equation}
        b_0Y^2\equiv ab_0Z^2\operatorname{mod}\mathfrak{m}_v.
    \end{equation}
    Since $b_0\in\mathcal{O}_v^\times$, we can multiply both sides by $b_0^{-1}$ to obtain $Y^2\equiv aZ^2\operatorname{mod}\mathfrak{m}_v$. Since $v(Y) = 0$ by assumption and since $v(a)=0$, $Z\not\equiv 0\operatorname{mod}\mathfrak{m}_v$ i.e. $Z\in\mathcal{O}_v^\times$ and so we find $(Y/Z)^2\equiv a\operatorname{mod}\mathfrak{m}_v$, which would mean that $a$ is a square modulo $\mathfrak{m}_v$, again a contradiction.
\end{proof}

Proposition \ref{main-prop-sec1} now follows from Lemmas \ref{d-exists}, \ref{neg-norm} and \ref{not-pure-quat-norm}.

\begin{remark}
    The condition $-1\in F^{\times 2}$ can be removed from the hypotheses of Propositions \ref{main-prop-sec1} and \ref{merk-psim} if $\operatorname{char }F = 0$ (i.e. when $F$ is a number field). Indeed, pick any $d\in F^{\times}\setminus F^{\times 2}$ and define $Q$ to be the quaternion algebra $(p,u)_F$, where $p\in\mathbb{Z}$ is an odd prime such that $F(\sqrt{d})$ embeds into $\mathbb{Q}_p$ and $u\in\mathbb{Z}$ is a nonsquare modulo $p$. One can verify that $d$ and $Q$ satisfy the conditions of Proposition ~\ref{main-prop-sec1}, from which Proposition \ref{merk-psim} then follows.
\end{remark}

\section{A criterion for the failure of weak approximation}

In this section we give a criterion for a connected reductive group to fail weak approximation over an arbitrary field. Results from this section will be used in \S 3 to construct explicit counterexamples to weak approximation among adjoint groups.

We start by showing that $R$-equivalence classes are open for groups over complete valued fields. Note that we do not require the valuations to be discrete for this result.

    \begin{lemma}\label{r-eq-open}
        Let $k$ be a complete valued field and let $\mathcal{G}$ be a connected reductive $k$-group. Then, the $R$-equivalence classes in $\mathcal{G}(k)$ are open (and closed) in the $k$-topology.
    \end{lemma}
    \begin{proof}
        It suffices to show that the $R$-equivalence class $[1]_R$ of the identity in $\mathcal{G}(k)$ is open, since the other $R$-equivalence classes are translates of $[1]_R$ by elements in $\mathcal{G}(k)$.

        Since $\mathcal{G}$ is reductive, the multiplication map $\mu\colon T_1\times\cdots\times T_m\to\mathcal{G}$ is dominant, where the $T_i$ are maximal tori in $\mathcal{G}$ (see \cite[Theorem 13.3.6 (ii), Corollary 7.6.4(ii), Corollary 2.2.7(i)]{LAGSpringer}). For each torus $T_i$, there exists a smooth surjective morphism $Q_i\to T_i$ from a rational $k$-torus $Q_i$. This can be seen for instance as a consequence of flasque resolution of tori (see \cite[\S 5]{colliot1977r}). Taking the product over all $i$, one obtains a smooth surjective morphism $\phi\colon Q\to T_1\times\cdots\times T_m$ of $k$-varieties, where $Q = Q_1\times\cdots\times Q_m$ is $k$-rational. The composition $b := \mu\circ\phi\colon Q\to\mathcal{G}$ is a smooth dominant morphism, and hence induces an open map (in the $k$-topology) $b:Q(k)\to\mathcal{G}(k)$ on the $k$-rational points \cite[Proposition 2.9(a)]{vcesnavivcius2015topology}. In particular, $b(Q(k))$ is open in $[1]_R\subset\mathcal{G}(k)$ implying that $[1]_R$ is open in $\mathcal{G}(k)$.
    \end{proof}

    \begin{remark}
    \begin{enumerate}
        \item In the case that $k$ is a local field, a similar result appears in \cite[Corollary 1.5]{kollar99}.
        \item Lemma \ref{r-eq-open} can in fact be shown in a number of ways. One method is to use the variety of maximal tori and the generic torus to obtain a smooth dominant morphism $f$ from a $k$-rational variety to $\mathcal{G}$, which would then be open on the $k$-rational points.
        Another way would be to use a finite separable morphism $g$ from a rational variety to $\mathcal{G}$ that exists as a consequence of unirationality of the connected reductive $k$-group $\mathcal{G}$. This morphism would then be generically \'{e}tale, hence \'{e}tale over an open subset of $\mathcal{G}$. One can then use the inverse function theorem for complete valued fields (see \cite[Corollary 2.1.1(ii)]{igusa2000introduction}) to conclude that $g$ is a local homeomorphism in the $k$-topology, which would again imply that $[1]_R$ is open in $\mathcal{G}(k)$.
    \end{enumerate}
\end{remark}

The following result now gives us a criterion for an arbitrary connected reductive group over a field to fail weak approximation.

\begin{theorem}\label{crit-wa}
    Let $\mathcal{G}$ be a connected reductive group over a field $k$ and let $S$ be a finite set of inequivalent discrete valuations of $k$. Assume that $\mathcal{G}$ has weak approximation with respect to $S$. Then the group homomorphism $\mathcal{G}(k)/R\to \prod_{v\in S}\mathcal{G}(k_v)/R$ is surjective. This is the map induced from the diagonal embedding $k\hookrightarrow\prod_{v\in S}k_v$ of $k$ into its completions.
\end{theorem}
\begin{proof}
    For $v\in S$ and an element $g\in\mathcal{G}(k_v)$, we will denote its $R$-equivalence class by $[g]_{R,v}$. This is a subset of $\mathcal{G}(k_v)$ and an element of the abstract group $\mathcal{G}(k_v)/R$.
            
            Suppose the map $\mathcal{G}(k)/R\to \prod_{v\in S}\mathcal{G}(k_v)/R$ is not surjective. Then we may pick an element $g = (g_v)\in \prod_{v\in S}\mathcal{G}(k_v)$ such that $([g_v]_{R,v})\in\prod_{v\in S}\mathcal{G}(k_v)/R$ does not lie in the image of $\mathcal{G}(k)/R$. This means that for for each $v\in S$, $g_v$ is not $R$-equivalent to any $h\in \mathcal{G}(k)$. In other words, for each $v$, $[g_v]_{R,v}\cap\mathcal{G}(k) = \emptyset$. Therefore, $(\prod_{v\in S}[g_v]_{R,v})\cap \mathcal{G}(k) = \emptyset$, where the intersection takes place in $\prod_{v\in S}\mathcal{G}(k_v)$. By Lemma \ref{r-eq-open}, $\prod_{v\in S}[g_v]_{R,v}$ is $v$-adically open in $\prod_{v\in S}\mathcal{G}(k_v)$. But by hypothesis, $\mathcal{G}(k)$ is dense and thus intersects every nonempty open subset nontrivially. We thus get a contradiction and the result follows.
\end{proof}

\section{Failure of weak approximation}
We are now ready to construct adjoint groups failing weak approximation with respect to a single valuation.
Fix a global field $F$ having characteristic different from $2$ and in which $-1$ is a square. Fix $G$ and $E$ as in Proposition ~\ref{merk-psim}.

\begin{lemma}\label{efg}
        For $F,G$ and $E$ as above, one has $G(F)/R = 1$ and $G(E)/R\neq~ 1$.
    \end{lemma}
    \begin{proof}
         We only need to show $G(F)/R = 1$. Since $F$ is either a totally imaginary number field or a global field of characteristic different from $2$, it has cohomological dimension $2$ (see \cite[Proposition 8.3.17, Theorem 10.1.11]{MR1737196}), and so we may use \cite[Theorem 3.7]{kulshrestha2008} to conclude $G(F)/R=~1$.
    \end{proof}

Our next goal is to show that the base change of $G$ to a suitable overfield of $F$ fails weak approximation. We begin with the following lemma:
    
    \begin{lemma}\label{val-fld}
        There is a field $K$ containing $F$, and a discrete valuation $v$ on $K$ with corresponding completion $K_v$, such that $G(K)/R = 1$ and $G(K_v)/R\neq ~1$.
    \end{lemma}
    \begin{proof}
        Choose a separating transcendence basis $\{e_1,\ldots,e_n\}$ of $E$ over $F$. Then the field $F' = F(e_1,\ldots,e_n)$ is a purely transcendental extension over $F$ while $E$ is a finite separable extension of $F'$.
        By the Primitive Element Theorem, $E = F'(\alpha)$ for some $\alpha\in E$.

        Let $f_\alpha\in F'[t]$ be the minimal polynomial of the element $\alpha$. Then we may endow the field $K = F'(t)$ with valuation $v$ induced by the prime ideal $(f_\alpha)$ in $F'[t]$. Observe that the corresponding residue field $\kappa_v$ is isomorphic to $E$. Denote by $K_v$ the completion of $K$ with respect to $v$. Observe that $G(K_v)/R\simeq G(\kappa_v)/R\simeq G(E)/R\neq 1$, where the first isomorphism comes from \cite[Theorem 8.14]{gille2025r}, applied with $A = \mathcal{O}_v\supseteq F$. On the other hand, $G(K)/R\simeq G(F')/R\simeq G(F)/R = 1$ by homotopy invariance (see \cite[Chapter 6, Proposition 2]{voskresenskii2011algebraic}).
    \end{proof}

Our main theorem now follows as an application of this criterion.
\begin{theorem}\label{conjecture-disproved}
    There exists an adjoint group $G$ defined over a field $K$, with a discrete valuation $v$ and corresponding completion $K_v$, such that the embedding $G(K)\to G(K_v)$ does not have $v$-adically dense image. In particular, $G$ fails weak approximation with respect to the singleton ~$\{v\}$.
\end{theorem}
\begin{proof}
    Let $F$ be any global field of characteristic different from $2$ in which $-1$ is a square. Let $G$ be as in Proposition \ref{merk-psim} and let $K,v$ and $K_v$ be as in Lemma \ref{val-fld}. Then $G(K)/R = 1$ and $G(K_v)/R \neq 1$, which implies that the map $G(K)/R\to G(K_v)/R$ is not surjective. By Theorem \ref{crit-wa}, $G$ fails weak approximation with respct to $\{v\}$.
\end{proof}

We conclude by providing an explicit counterexample:

\begin{example}\label{example}
    Consider $F = \mathbb{Q}(i)$ and the quaternion algebra $Q = (2,5)_F$ over $F$. This is the unique (up to isomorphism) quaternion algebra over $F$ ramified exactly at the places $(2+i)$ and $(2-i)$. One may check that $d=6$ satisfies the conditions in Proposition \ref{main-prop-sec1}. Then, by Proposition \ref{merk-psim}, the algebra $A = \operatorname{M}_3(Q)$ has an involution $\sigma$ such that $G = \textbf{PSim}_+(A,\sigma)$ is not universally $R$-trivial.

    As in the proof of Theorem \ref{conjecture-disproved}, Lemma \ref{val-fld} then implies that $G_K$ fails weak approximation for some extension $K$ of $F$. Following Merkurjev's construction in \cite[Theorem 2, Lemma 8]{merkurjev1996r}, one can find such a $K$ with  $\operatorname{tr.deg}_FK\leq 13$.

    More generally, for each $A_n = M_n(Q)$ with $n\geq 3$ odd, there is an involution $\sigma_n$ such that $G_n = \textbf{PSim}_+(A_n,\sigma_n)$ does not have weak approximation when base changed to some field $K$ of transcendence degree at most $3n+4$ over $F = \mathbb{Q}(i)$.
\end{example}

\section*{Acknowledgments}
I am indebted to my advisor, J. Hartmann, for her guidance, careful reading of multiple drafts, and numerous suggestions that substantially improved this work. I am grateful to D. Krashen and D. Harbater for helpful conversations. I thank M. Ji and X. Meng for assistance with preliminary computational work that helped guide some results. Lastly, I thank K. Becher, V. Chernousov, J.-L. Colliot-Th\'{e}l\`{e}ne, J. Koll\'{a}r, I. Rapinchuk and N.Q. Th\u{a}\'{n}g for constructive comments on an earlier version of the manuscript. This material is based upon work supported in part by the U.S. NSF under Grant Numbers DMS-2102987 and DMS-2402367.

     \bibliographystyle{alpha}
    \bibliography{bibliography}{}

\end{document}